\documentclass[a4paper,10pt,reqno, english]{amsart}

\usepackage{amsmath,amssymb,amscd,amsthm,amsfonts}
\usepackage{graphicx,subfigure}
\usepackage{hyperref}
\usepackage{dsfont}
\usepackage[nobysame, alphabetic]{amsrefs}
\usepackage{tikz}
\usepackage[capitalise]{cleveref}
\usepackage{mathrsfs}

\newtheorem{theorem}{Theorem}

\newtheorem{corollary}{Corollary}

\newtheorem{definition}{Definition}
\newtheorem{remark}{Remark}

\def\zz{\mathds{Z}}
\def\kk{\mathcal{K}}
\def\LL{\mathscr{L}}

\def\rr{\mathds{R}}

\DeclareMathOperator{\conv}{conv}
\DeclareMathOperator*{\bigast}{\mathop{\scalebox{1.5}{$\ast$}}}

\title{Selection-structure generalizations of the Borsuk--Ulam theorem}

\hypersetup{
  pdftitle={Selection-structure generalizations of the Borsuk--Ulam theorem},
  pdfauthor={Pablo Sober\'on}
}

\author[Sober\'on]{Pablo Sober\'on}\address{Baruch College, City University of New York, One Bernard Baruch Way, New York, NY 10010, United States} 
\email{psoberon@gc.cuny.edu}

\thanks{
The research of P. Sober\'on is supported by NSF CAREER grant DMS-237324 and a PSC-CUNY Trad B award.}



\begin{document}

\begin{abstract}
We prove Borsuk--Ulam-type
results governed by selection structures.  Selection structures
extend the matroidal framework for colorful theorems in discrete geometry and include
non-matroidal examples such as chessboard complexes.  

Motivated by Frick and Wellner's Radon-type strengthening of
Fan's theorem and its colorful variants, we prove
selection-structure analogues whose conclusions are determined by
Radon partitions.  We also prove a prime-power
selection-structure covering version of Volovikov's theorem,
governed by Tverberg partitions.  We include
applications to fair partitions, including selection-structure
versions of the ham sandwich and necklace splitting theorems.
\end{abstract}

\maketitle

\section{Introduction}

The Borsuk--Ulam theorem stands as one of the most emblematic theorems in topological combinatorics \cite{Borsuk1933}.  The applications of this result are remarkable both in scope and elegance \cites{Matousek2003, Zivaljevic2017}.  As a consequence, there have been many generalizations of the Borsuk--Ulam theorem.  Consider the following theorem by Fan.

\begin{theorem}[Fan 1952 \cite{Fan1952}]\label{thm:fan-1952}
    Let $m,d$ be positive integers and $A_1,\dots,A_m$ be closed subsets of $S^d$ such that $A_i \cap (-A_i) = \emptyset$ for all $i \in [m]=\{1,\dots,m\}$ and such that $S^d = \bigcup_{i \in [m]} (A_i \cup (-A_i))$.  Then, there exist $d+1$ numbers $i_1 < i_2 < \ldots < i_{d+1}$ in $[m]$ such that 
    \[
    \bigcap_{j=1}^{d+1} (-1)^jA_{i_j} \neq \emptyset.
    \]
\end{theorem}

The reduction of Fan's theorem to Borsuk--Ulam is simple.  If we have an odd map $f:S^d \to \rr^{m}$ without zeros, we can declare $x \in S^d$ to be in $A_i$ if and only if the $i$-th coordinate of $f(x)$ is positive and has the largest absolute value among all coordinates of $f(x)$.  Equivalently, $x \in A_i$ if and only if $f(x)_i = \|f(x)\|_{\infty}$.  The sets $A_1,\dots, A_m$ satisfy the condition of the theorem, and the conclusion is impossible to achieve if $m \le d$.  While the Borsuk--Ulam theorem only provides conditions for such odd maps without zeros to exist, Fan's theorem provides interesting consequences when they do exist, even when $m-d$ is large.

Recently, Frick and Wellner proved several extensions of Fan's theorem, building on their earlier results \cite{Frick2025}, where the coefficients involved are determined by convexity conditions.  The result below implies Fan's theorem if the points $x_1,\dots, x_m$ are placed on the moment curve in $\rr^{d-1}$.

\begin{theorem}[Frick, Wellner 2025 \cite{Frick2025arxiv}]
    Let $m,d$ be positive integers, $\{x_1,\dots,x_m\}$ be a set of $m$ points in $\rr^{d-1}$, and $A_1,\dots, A_{m}$ be closed subsets of $S^d$.  Suppose that for each $i \in [m]$ we have $A_i \cap (-A_i) = \emptyset$, and that $S^d = \bigcup_{i \in [m]}(A_i \cup (-A_i))$.  Then, there exist two disjoint subsets $P_+, P_- \subset [m]$ such that $\conv \{x_i : i \in P_+\}\cap \conv\{x_i : i \in P_-\}\neq \emptyset$ and such that
    \[
    \left(\bigcap_{i \in P_+}A_i\right) \cap \left(\bigcap_{i \in P_-}-A_i\right)\neq \emptyset.
    \]
\end{theorem}

Additionally, Frick and Wellner proved colorful versions of their generalizations of Fan's theorem.  This aligns with the trend in discrete geometry and topological combinatorics, where many classic results have a ``colorful'' generalization.  Examples include colorful versions of Helly's theorem \cite{Barany1982}, Caratheodory's theorem \cite{Barany1982}, and the KKM theorem \cite{Gale1984}.  Colorful results, despite their playful name, are vastly more powerful tools than their standard counterpart.  The colorful conditions in many of these theorems can be replaced by the independence complex of a matroid \cites{Kalai2005, Holmsen2016, mcginnis2024arxiv}.  In this manuscript, motivated by the colorful Fan-type results of Frick and Wellner, we prove selection-structure analogues governed by Radon partitions. These include matroidal, sparse, and nonmatroidal
specializations.

A selection structure is a simplicial complex together with an upward-closed family of sets that can be used to extend colorful results beyond matroids.  In an earlier work, the author introduced selection structures and showed that many colorful results can be generalized to that level, including Helly's theorem, Carath\'eodory's theorem, and the KKM theorem \cite{soberon2026arxiv}.  Selection structures allow us to impose the same colorful conditions that we would get from matroids, but also include non-matroidal complexes such as chessboard complexes and matching complexes.

To state our main result, we first define selection structures.

\begin{definition}
    Let $W$ be a finite set and let $\kk$ be a simplicial complex on $W$.  Let $D=\{D_w : w \in W\}$ be an indexed family of nonempty finite sets, one for each element of $W$.  We refer to the sets in $D$ as fibers.  For $U \subseteq W$, we define $\kk[U] = \{\sigma \in \kk : \sigma \subseteq U\}$ to be the restriction of $\kk$ to $U$.  Then, we define the space
    \[
    \kk(U;D) = \bigcup_{\sigma \in \kk[U]} \left(\bigast_{w \in \sigma} D_w\right)
    \]
    where $\bigast$ denotes an iterated join.
\end{definition}

The complex $\kk(U;D)$ is essentially a copy of $\kk[U]$ in which every vertex $w$ is replaced by its fiber $D_w$, see \cref{fig:expansion}.  Then, every set of vertices of $\kk(U;D)$ from different fibers that maps to a face of $\kk[U]$ is also a face of $\kk(U;D)$.

\begin{figure}
    \centering
    \includegraphics[width=\linewidth]{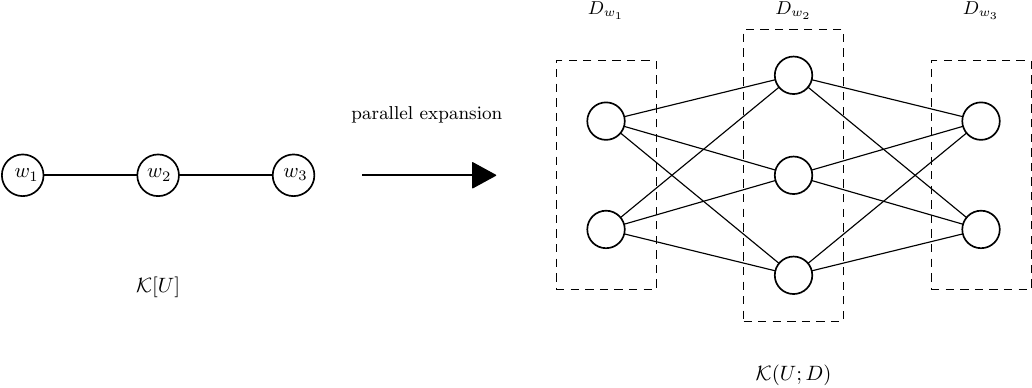}
    \caption{An example of a parallel expansion of a path with three vertices.}
    \label{fig:expansion}
\end{figure}

With this, we are ready to define $d$-admissible selection structures.

\begin{definition}
    Let $d$ be a positive integer.  Let $W$ be a finite set, $\kk$ be a simplicial complex on $W$, and $\LL$ be a nonempty upward-closed family of subsets of $W$.  We say that $\Sigma = (\kk, \LL)$ is a $d$-admissible selection structure if the following two conditions hold.
    \begin{itemize}
        \item For any indexed family $D = \{D_w : w \in W\}$ of nonempty finite sets, and every $U \in \LL$, the complex $\kk(U;D)$ is $(d-1)$-connected.
        \item For every $U \in \LL$, the facets of $\kk[U]$ are also facets of $\kk$.
    \end{itemize}
\end{definition}

Our main result is the following theorem.

\begin{theorem}[Selection-structure Fan--Radon theorem]\label{thm:main-selection-frickwellner}
    Let $d, m$ be positive integers, $W$ be a finite set, and $\Sigma =(\kk,\LL)$ be a $d$-admissible selection structure on $W$.  For every $w \in W$, let $X^w = \{x^w_1,\dots, x^w_m\} \subset \rr^{d-1}$ be a set of points.  For every $w\in W$ and $i\in [m]$, let $A^w_i$ be a closed subset of $S^d$.  Assume that the following two conditions hold:
    \begin{itemize}
        \item For all $W' \subseteq W$ such that $W \setminus W' \not\in \LL$, we have
        \( \displaystyle
        S^d = \bigcup_{w \in W'} \bigcup_{i=1}^{m} \left( A_i^w \cup (-A_i^w)\right), 
        \)
        \item For every pair of different elements $u, w \in W$ such that $\{u,w\} \in \kk$ and every $i \in [m]$, we have $A_i^u \cap (-A_i^w) = \emptyset$.
    \end{itemize}
    Then there exist $x \in S^d$, two nonempty disjoint sets $P_+,P_- \subseteq W \times [m]$ with the following properties.  First,
    \(
    \conv \left(x^{w}_i : (w,i) \in P_+ \right)
    \cap \conv \left(x^{w}_i : (w,i) \in P_-\right) \neq \emptyset.
    \)
    Second, 
    \[
    x \in \left( \bigcap_{(w,i) \in P_+} A_i^{w}\right)\cap \left(\bigcap_{(w,i) \in P_-} \left(- A_i^{w}\right)\right).
    \]
    Third, the projection $\pi_1: P_+ \cup P_- \to W$ onto the first coordinate is injective and $\pi_1 (P_+ \cup P_-)$ is a facet of $\kk$.  The projection $\pi_2: W \times [m] \to [m]$ satisfies $\pi_2(P_+) \cap \pi_2(P_-) = \emptyset$.

\end{theorem}

The conclusion of intersecting pairs of convex sets brings Fan's theorem much closer to Radon's theorem \cite{Radon:1921vh}.  If we align with the notation from the results above, Radon's theorem guarantees that \textit{for any $d+1$ points in $\rr^{d-1}$, there exists a partition of them into two sets whose convex hulls intersect}.  One key generalization of Radon's theorem is Tverberg's theorem \cite{Tverberg1966}, which is a central result in discrete geometry \cites{Blagojevic2017, Barany2018, DeLoera2019}.  It states that \textit{for any $(r-1)d+1$ points in $\rr^{d-1}$ there exists a partition of them into $r$ sets whose convex hulls intersect}.

Just as Radon's theorem can be proved with the Borsuk--Ulam theorem \cite{Bajmoczy:1979bj}, the main topological tool behind many Tverberg-type results is a theorem by Volovikov about the existence of equivariant maps between two spaces \cite{Volovikov1996}.  Volovikov's theorem deals with maps equivariant with respect to the group $(\zz_p)^a$ for $p$ a prime and $a$ a positive integer, and generalizes the Borsuk--Ulam theorem.

We prove a selection-structure covering generalization of Volovikov's theorem in \cref{thm:selection-volovikov}.  The conclusion in that theorem is imposed by $r$-tuples of convex sets intersecting, with $r$ being a prime power.  When $r=2$, \cref{thm:selection-volovikov} specializes to \cref{thm:main-selection-frickwellner}.

There exist many extensions of Fan's theorem and Tucker's lemma beyond $\zz_2$-actions (see, e.g., \cites{Ziegler2002, Meunier2019a, Longueville2006}).  Our results show that the extensions to selection structures are also feasible in this context.

We prove \cref{thm:main-selection-frickwellner} in \cref{sec:selection-structure}.  We show several specializations of this theorem in \cref{sec:secializations}.  Then, we state and prove the Volovikov-style generalization in \cref{sec:ss-volovikov}.  Finally, we prove the applications to mass partitions and necklace splitting in \cref{sec:applications}.

\section{Selection-structure generalizations of Fan's theorem}\label{sec:selection-structure}

Let us begin by proving \cref{thm:main-selection-frickwellner}.

\begin{proof}
    Let $\delta>0$ be such that for all distinct $u,w \in W$ with $\{u,w\}\in \kk$ and $i \in [m]$, no subset of $S^d$ of diameter at most $\delta$ intersects both $A^u_i$ and $-A^w_i$.

    Let $T$ be a centrally symmetric triangulation of $S^d$ with mesh strictly smaller than $\delta$, and let $E = \{\pm 1, \dots, \pm m\}$, with a natural action of $\zz_2 =\{+1,-1\}$ with multiplication.

    For a vertex $v$ of $T$, let $D_w(v) =\{+i : v \in A^w_i\} \cup \{-i: v \in -A^w_i\}$.  Let $U(v) = \{w \in W: D_w(v) \neq \emptyset\}$.  The set $U(v)$ is the set of elements $w \in W$ such that $v$ is contained in $A^w_i \cup (-A^w_i)$ for some $i \in [m]$.

    If we had $U(v) \not\in \LL$, notice that by taking $W' = W\setminus U(v)$, we would have
    \[
    v \not\in \bigcup_{w \in W'} \bigcup_{i=1}^{m} \left( A_i^w \cup (-A_i^w)\right),
    \]
    contradicting the assumptions.  Therefore, $U(v) \in \LL$ for all vertices $v$.

    For a simplex $\tau$ of $T$, let $V(\tau)$ be its set of vertices, and define
    \[
    D_w(\tau) = \bigcup_{v \in V(\tau)} D_w(v) \qquad U(\tau) = \bigcup_{v \in V(\tau)}U(v).
    \]
    Since $\LL$ is upward-closed and $U(v) \in \LL$ for all $v$, we have $U(\tau) \in \LL$ too.  Finally, let $D(\tau) = \{D_w(\tau) : w \in U(\tau)\}$.  The set $D(\tau)$ is taken as a multiset, as we want one fiber for each $w \in W$, and we may have repeated fibers.  Consider the set
    \[
    \Gamma (\tau) = \kk (U(\tau), D(\tau)).
    \]
    This is the parallel expansion of $\kk[U(\tau)]$ by the (signed) fibers $D(\tau)$.

    The goal now is to construct a map
    \begin{align*}
        \psi : S^d=|T| & \to |\kk(W; E)| \qquad \mbox{such that for all } \tau \in T \\
        \psi (\tau) & \subseteq |\Gamma(\tau)|.
    \end{align*}

    This is done inductively on skeleta.  For each vertex $v$ of $T$, we define $\psi(v)$ as a vertex of $\Gamma(v)$ arbitrarily and equivariantly with respect to the $\zz_2$ action.  If the map $\psi$ has been defined on the $(r-1)$-skeleton and $\tau$ is an $r$-simplex, notice that $\psi (\partial \tau) \subset |\Gamma(\tau)|$.  Since $\Gamma(\tau)$ is $(d-1)$-connected, the map $\psi$ can be extended over $\tau$, since $r \le d$.  We extend the map on $-\tau$ equivariantly with respect to the $\zz_2$ action.  Note that $\sigma \subseteq \tau$ implies $\Gamma (\sigma) \subseteq \Gamma (\tau)$ and that $\Gamma(-\tau) = - \Gamma(\tau)$.

    Now we define a second map
    \begin{align*}
        \Phi: |\kk(W; E)| & \to \rr^{d} = \rr^{d-1}\times \rr 
    \end{align*}
    by defining it on the vertices $\Phi (w,+i) = (x^w_i,1)$ and $\Phi (w,-i) = (-x^w_i,-1)$, and extending linearly on the faces of $|\kk(W; E)|$.  The map \(F= \Phi \circ \psi : S^d  \to \rr^d\) is continuous and odd.  By the Borsuk--Ulam theorem, it must have a zero.

    Let $z$ be a zero of $F$.  Let $\tau$ be a face of $T$ that contains $z$, and $\eta$ be a containment-maximal face of $\Gamma(\tau)$ that contains $\psi (z)$.  Let $P_+ = \{(w,i) \in W \times [m]: (w,+i) \in \eta \}$ and $P_- = \{(w,i)\in W \times [m]: (w,-i) \in \eta\}$.
    It is a basic exercise from convexity that $0 \in \conv \Phi(\eta)$ if and only if
    \[
    \conv \{x^w_i: (w,i) \in P_+\} \cap \conv \{x^w_i : (w,i) \in P_-\} \neq \emptyset.
    \]

    First, we show that $\pi_1: P_+\cup P_- \to W$ is injective.  This is because the faces of $|\kk(W; E)|$ have at most one vertex from each fiber.  We also have $\pi_1(P_+) \cup \pi_1 (P_-) = \pi_1(\eta)$, which by definition is a face of $\kk$.  Moreover, it is a facet of $\kk[U(\tau)]$, which by the definition of selection structures is also a facet of $\kk$.

    Now let us show that $\pi_2(P_+) \cap \pi_2(P_-) = \emptyset$.  If we had $(u,+i) \in \eta$ and $(w, -i) \in \eta$, as stated above we must have $u \neq w$.  Since $(u,+i) \in \eta$ and $(w, -i) \in \eta$, we have $\{u,w\} \in \Gamma(\tau)$, which in turn implies $\{u,w\} \in \kk$.  Since $\eta \subseteq \Gamma(\tau)$, there are vertices $p,q \in V(\tau)$ with $p \in A^u_i$ and $q \in (-A^w_{i})$.  However, the diameter of $\tau$ is less than $\delta$, contradicting the choice of $\delta$.  Therefore no index $i$ appears with both signs.  


    
    What we have shown is that, for every $(w,i) \in P_+ \cup P_-$, some vertex of $\tau$ is in the corresponding set $A^{w}_i$ or $-A^{w}_i$ (depending on whether $(w,i) \in P_+$ or $(w,i) \in P_-$).  If we take a sequence of triangulations $T_{\alpha}$ with mesh tending to $0$, the finiteness of $W$ and the compactness of $S^d$ imply by a standard argument that the same holds for some point $x$ in the limit, which gives us the desired conclusion.

\end{proof}

\begin{corollary}[Diagonal version]\label{cor:diagonal}
    In \cref{thm:main-selection-frickwellner}, assume in addition that $x_i^w=x_i$ for all
    $w\in W$ and all $i\in[m]$ (i.e., all $m$-tuples $X^w$ are equal).  Then, we can replace the condition in the conclusion that $\pi_1(P_+\cup P_-)$ is a facet of $\kk$ by the following two conditions.  First, $\pi_1(P_+ \cup P_-)$ is a face of $\kk$.  Second, the projection onto the second coordinate $\pi_2: P_+ \cup P_- \to [m]$ is injective. 
\end{corollary}

\begin{proof}
    In the construction of $\eta$, if we have $(u,+i)\in \eta$ and $(w,+i) \in \eta$ for some $u,w \in W$ and $i \in [m]$, we can remove one of them and have the other absorb the coefficient in the convex combination that showed $0 \in \Phi (\eta)$.  Since $\Phi (w,+i) = \Phi (u,+i)$, this causes no problems.  Therefore, the projection onto $[m]$ restricted to each set $P_+$ or $P_-$ is injective.  We have already shown that there are no repeated second coordinates in the two sets, which implies the desired conclusion.
\end{proof}

\section{Different specializations of \cref{thm:main-selection-frickwellner}}\label{sec:secializations}

Let us show how we recover results similar to the colorful result from Frick and Wellner,  \cite{Frick2025arxiv}*{Thm. 2.6}.   

\begin{proof}[Simplex specialization]
    Let $W = [d+1]$, $\kk=2^W$, and $\LL = \{W\}$.  This gives us a $d$-admissible selection structure since $\kk(W,D) = \bigast_{w \in W}D_w$, and the join of $d+1$ sets is $(d-1)$-connected.  The facet condition is trivial.  We have $d+1$ families of sets $\{A^w_i: i \in [m]\}$.  The first condition translates to $\bigcup_{i\in [m]}\left(A^w_i \cup (-A^w_i)\right)=S^d$ for all $w \in [d+1]$.  The second condition translates to $A^u_i \cap (-A^w_i) = \emptyset$ for all $u,w \in [d+1]$.  This are exactly the hypotheses for \cite{Frick2025arxiv}*{Thm. 2.6}.  When we apply \cref{thm:main-selection-frickwellner}, there is only one maximal face $B =[d+1]$ of $\kk$, which implies that $\pi_1$ must be a bijective function on $P_+ \cup P_-$.  For the projection $\pi_2$ onto the second coordinate, the two sets $P_+$ and $P_-$ satisfy $\pi_2(P_+) \cap \pi_2(P_-) = \emptyset$, so they do not share subindices $i\in [m]$.
\end{proof}

\begin{remark}
In the simplex specialization above, the projection
\( \pi_1:P_+\cup P_-\rightarrow[d+1] \)
is bijective, while
\(
\pi_2(P_+)\cap\pi_2(P_-)=\emptyset\).  The conclusion of \cite{Frick2025arxiv}*{Theorem 2.6}
additionally requires \(\pi_2:P_+\cup P_-\rightarrow[m]\)  to be injective. Our simplex specialization differs from
their stated conclusion by this additional label-injectivity
requirement.

In the diagonal setting of \cref{cor:diagonal}, repeated
same-sign labels can be merged, and the additional injectivity
can be imposed.
\end{remark}

It is interesting that the Frick--Wellner colorful result is related to the simplex specialization of \cref{thm:main-selection-frickwellner}.  For colorful results that have a selection-structure generalization, the standard colorful theorem is recovered when $\kk$ is the independence complex of a partition matroid of rank $d+1$, and $\LL$ is the family of sets of full rank.  This suggests a ``doubly-colorful'' Fan--Radon covering theorems.

Before stating the doubly-colorful version, let us state what the natural sparse version of the Fan--Radon theorem would be.

\begin{corollary}[Sparse Fan--Radon theorem]
    Let $d, m, n$ be positive integers such that $n \ge d+1$.  For each integer $w \in [n]$, let $X^w = \{x^w_1,\dots , x^w_m\} \subset \rr^{d-1}$ be a set of $m$ points.  For every $w \in [n]$ and $i \in [m]$, let $A^w_i$ be a closed subset of $S^d$.  Assume that the following two conditions hold.
    \begin{itemize}
        \item For all $W'\subseteq [n]$ such that $|W'| \ge n-d$, we have
        \[
        S^d = \bigcup_{w \in W'}\bigcup_{i=1}^m (A^w_i \cup (-A^w_i)).
        \]
        \item For every pair of different numbers $u,w \in [n]$ and every $i \in [m]$, we have $A^u_i \cap (-A^w_i) = \emptyset$.
    \end{itemize}

    Then, we can find two disjoint sets $P_+, P_-$ of $[n] \times [m]$ such that $|P_+| + |P_-| = d+1$, the projection $\pi_1: P_+ \cup P_- \to [n]$ is injective, the projection $\pi_2:P_+ \cup P_- \to [m]$ satisfies $\pi_2 (P_+) \cap \pi_2 (P_-) = \emptyset$, and
    \[
    \conv \{x^w_i : (w,i) \in P_+\} \cap \conv \{x^w_i : (w,i) \in P_-\}\neq \emptyset.
    \]
    Additionally, there exists $x \in S^d$ with
    \[
    x \in \left( \bigcap_{(w,i) \in P_+}A^w_i \right) \cap \left( \bigcap_{(w,i) \in P_-}(-A^w_i) \right).
    \]
\end{corollary}

\begin{proof}
    Apply \cref{thm:main-selection-frickwellner} when $W = [n]$, $\kk$ is the $d$-skeleton of $2^{[n]}$, and $\LL$ is the family of subsets of $[n]$ with at least $d+1$ elements. 
\end{proof}

The doubly-colorful Fan--Radon theorem is the following.

\begin{corollary}[doubly-colorful Fan--Radon theorem]\label{cor:doubly-colorful}
    Let $d, m$ be positive integers, $W$ be a finite set partitioned into $d+1$ sets $W=W_1\sqcup \dots \sqcup W_{d+1}$.  For each $w \in W$, let $X^w = \{x^w_1,\dots, x^w_m\} \subset \rr^{d-1}$ be a set of points.  For every $w\in W$ and $i\in [m]$, let $A^w_i$ be a closed subset of $S^d$.  Assume that the following two conditions hold:
    \begin{itemize}
        \item For all $j \in [d+1]$, we have
        \[
        S^d = \bigcup_{w \in W_j} \bigcup_{i=1}^{m} \left( A_i^w \cup (-A_i^w)\right), 
        \]
        \item For every pair of elements $u, w \in W$ in different sets of the partition and every $i \in [m]$, we have $A_i^u \cap (-A_i^w) = \emptyset$.
    \end{itemize}
    Then there exist $x \in S^d$, two nonempty disjoint sets $P_+,P_- \subseteq W \times [m]$ with the following properties.  First,
    \(
    \conv \left(x^{w}_i : (w,i) \in P_+ \right)
    \cap \conv \left(x^{w}_i : (w,i) \in P_-\right) \neq \emptyset.
    \)
    Second, 
    \[
    x \in \left( \bigcap_{(w,i) \in P_+} A_i^{w}\right)\cap \left(\bigcap_{(w,i) \in P_-} \left(- A_i^{w}\right)\right).
    \]
    Third, the projection $\pi_1: P_+ \cup P_- \to W$ onto the first coordinate is injective and $\pi_1(P_+ \cup P_-)$ has exactly one element from each $W_j$.  The projection $\pi_2: W \times [m] \to [m]$ satisfies that $\pi_2(P_+) \cap \pi_2(P_-) = \emptyset$.
\end{corollary}

\begin{proof}
    Let $\kk$ be the simplicial complex of sets with at most one element from each $W_j$, and $\LL$ the family of sets with at least one element from each $W_j$.  This is a $d$-admissible selection structure \cite{soberon2026arxiv}.  One reason is that this corresponds to a partition matroid of rank $d+1$ as in \cref{def:selection-matroid}.  Apply \cref{thm:main-selection-frickwellner} to conclude.
\end{proof}

In the case when all the $m$-tuples $X^w$ are equal, the theorem is equivalent to the simplex specialization, as we can replace the family of sets $\{A^w_i : w \in W_j\}$ by $\bigcup_{w \in W_j} A^w_i$.  Then, an application of the simplex specialization of \cref{thm:main-selection-frickwellner} gives the same result.  Once we allow for different $m$-tuples, the double coloring gives new results.

To illustrate how \cref{thm:main-selection-frickwellner} applies to a broader set of examples, let us describe the matroid and the chessboard complex examples.

Given a matroid $M$ of rank at least $d+1$, the author has shown that the following is a $d$-admissible selection structure \cite{soberon2026arxiv}.  This is a direct consequence of the shellability of matroid independence complexes \cite{Bjoerner1992}.

\begin{definition}\label{def:selection-matroid}
    Let $d$ be a positive integer, $W$ be a finite set and $M$ be the independence complex of a matroid of rank greater than or equal to $d+1$ on $W$.  Let $\rho$ be the rank function of $M$.  Consider the families
    \begin{align*}
    \kk_{M,d} &= M^{(d)}, \mbox{ the $d$-skeleton of $M$, and} \\
    \LL_{{M,d}} &= \{U \subseteq W: \rho (U) \ge d+1\}.
    \end{align*}
    Then, $\Sigma = (\kk, \LL)$ is a $d$-admissible selection structure.
\end{definition}

\begin{corollary}[Matroid Fan--Radon theorem]
    Let $d, m$ be positive integers, $W$ be a finite set, and $M$ the independence complex of a matroid of rank greater than or equal to $d+1$ on $W$ with rank function $\rho$.  For each $w \in W$, let $X^w = \{x^w_1,\dots, x^w_m\} \subset \rr^{d-1}$ be a set of points.  For every $w\in W$ and $i\in [m]$, let $A^w_i$ be a closed subset of $S^d$.  Assume that the following two conditions hold:
    \begin{itemize}
        \item For all $W' \subseteq W$ such that $\rho (W \setminus W') \le d$, we have
        \[
        S^d = \bigcup_{w \in W'} \bigcup_{i=1}^{m} \left( A_i^w \cup (-A_i^w)\right), 
        \]
        \item For every pair of elements $u, w \in W$ such that $\{u,w\}\in M$ and every $i \in [m]$, we have $A_i^u \cap (-A_i^w) = \emptyset$.
    \end{itemize}
    Then there exist $x \in S^d$, two nonempty disjoint sets $P_+,P_- \subseteq W \times [m]$ with the following properties.  First,
    \(
    \conv \left(x^{w}_i : (w,i) \in P_+ \right)
    \cap \conv \left(x^{w}_i : (w,i) \in P_-\right) \neq \emptyset.
    \)
    Second, 
    \[
    x \in \left( \bigcap_{(w,i) \in P_+} A_i^{w}\right)\cap \left(\bigcap_{(w,i) \in P_-} \left(- A_i^{w}\right)\right).
    \]
    Third, the projection $\pi_1: P_+ \cup P_- \to W$ onto the first coordinate is injective and $\pi_1 (P_+ \cup P_-) \in M$ and has rank $d+1$.  The projection $\pi_2: W \times [m] \to [m]$ satisfies that $\pi_2(P_+) \cap \pi_2(P_-) = \emptyset$.
\end{corollary}

\begin{proof}
    Apply \cref{thm:main-selection-frickwellner} to the selection structure from \cref{def:selection-matroid}.
\end{proof}

When $M$ is a partition matroid, we recover \cref{cor:doubly-colorful}.  Given $W = X \times Y$, the chessboard complex on $W$ is the family of sets $W' \subseteq W$ that have at most one element in each row and at most one element in each column (i.e., the two coordinate projections restricted to $W'$ are injective functions).  Chessboard complexes appear naturally in combinatorial problems, and their topological properties are well studied \cite{Bjoerner1994}.

To define our selection structure for chessboard complexes, the most delicate point is the definition of $\LL$.  To do this, consider the following definition of an $(a,b,d+1)$-stair shape, introduced by Ziegler \cite{Ziegler1994}.

\begin{definition}\label{def:stair-shape}
    Let $a,b,d$ be integers such that $\nu(a,b) = \min\{a,b,\lfloor \frac{a+b+1}{3}\rfloor \} \ge d+1$, $a \le 2d+1$, and $b \le 2d+1$.  The $(a,b,d+1)$-stair is the set of squares $(i,j) \in [a]\times[b]$ such that
    \[
    a-(2d+1) \le j-i \le (2d+1)-b.
    \]
\end{definition}

\begin{figure}
    \centering
    \includegraphics[width=0.9\linewidth]{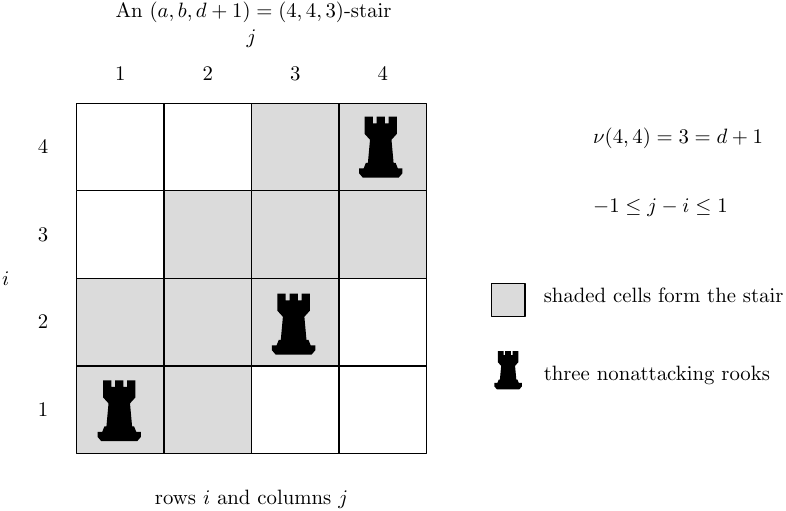}
    \caption{An example of a $(4,4,3)$-stair and three non-attacking rooks.  The placement of the rooks make a face of the chessboard complex.}
    \label{fig:stairs}
\end{figure}

See \cref{fig:stairs} for an illustration of a stair.  The existence of selection structures for chessboard complexes was also established in \cite{soberon2026arxiv}.

\begin{definition}\label{def:selection-chessboard}
    Let $d$ be a positive integer, let $X, Y$ be finite sets such that $\nu(|X|,|Y|) \ge d+1$, and let $W= X \times Y$.  Let $\kk_{\operatorname{st}, d}$ be the $d$-skeleton of the chessboard complex of $W$.  Let $\LL_{\operatorname{st}, d}$ be the family of sets $U$ such that there exist $a,b$ satisfying the conditions of \cref{def:stair-shape} and a set $U' \subseteq U$ isomorphic to the $(a,b,d+1)$-stair.  Then, $\Sigma=(\kk_{\operatorname{st}, d}, \LL_{\operatorname{st}, d})$ is a $d$-admissible selection structure.
\end{definition}

This gives us the following chessboard complex version of Fan's theorem.

\begin{corollary}[Chessboard complex Fan--Radon theorem]
    Let $d, m$ be positive integers, $X,Y$ be finite sets such that $\nu(|X|,|Y|) \ge d+1$, and $W = X \times Y$.  For each $w \in W$, let $X^w = \{x^w_1,\dots, x^w_m\} \subset \rr^{d-1}$ be a set of points.  For every $w\in W$ and $i\in [m]$, let $A^w_i$ be a closed subset of $S^d$.  Assume that the following two conditions hold:
    \begin{itemize}
        \item For all $W' \subseteq W$ such that $W'$ intersects every set isomorphic to some $(a,b,d+1)$-stair, where $a \le |X|, b\le |Y|$ are integers that satisfy the conditions of \cref{def:stair-shape}, we have
        \[
        S^d = \bigcup_{w \in W'} \bigcup_{i=1}^{m} \left( A_i^w \cup (-A_i^w)\right), 
        \]
        \item For every pair of elements $u=(x,y), w=(x',y') \in W$ such that $x \neq x'$ and $y \neq y'$, and every $i \in [m]$, we have $A_i^u \cap (-A_i^w) = \emptyset$.
    \end{itemize}
    Then there exist $x \in S^d$, two nonempty disjoint sets $P_+,P_- \subseteq W \times [m]$ with the following properties.  First,
    \(
    \conv \left(x^{w}_i : (w,i) \in P_+ \right)
    \cap \conv \left(x^{w}_i : (w,i) \in P_-\right) \neq \emptyset.
    \)
    Second, 
    \[
    x \in \left( \bigcap_{(w,i) \in P_+} A_i^{w}\right)\cap \left(\bigcap_{(w,i) \in P_-} \left(- A_i^{w}\right)\right).
    \]
    Third, the projection $\pi_1: P_+ \cup P_- \to W = X \times Y$ is injective and $\pi_1 (P_+ \cup P_-)$ corresponds to a placement of $d+1$ non-taking rooks.  The projection $\pi_2: W \times [m] \to [m]$ satisfies $\pi_2(P_+) \cap \pi_2(P_-) = \emptyset$.
\end{corollary}

\begin{proof}
    Use \cref{thm:main-selection-frickwellner} with the selection structure from \cref{def:selection-chessboard}.
\end{proof}

\section{Tverberg-type Volovikov theorem}\label{sec:ss-volovikov}

Let $G$ be a group.  For some positive integer $n$, we denote by $G^{*n}$ the $n$-fold join of $G$.  This is an $(n-2)$-connected space with an action of $G$, where $h \cdot (\sum_{g \in G} \alpha_g g) = \sum_{g \in G} \alpha_g (hg)$.  We denote by $W_G$ the $(|G|-1)$-dimensional real space
\[
W_G = \{x \in \rr^{G}: \sum_{g \in G} x_g = 0\},
\]
which also has a natural action of $G$.

\begin{theorem}[Volovikov 1996 \cite{Volovikov1996}]
    Let $d,a$ be positive integers, let $p$ be a prime, and let $r=p^a$, $n=(r-1)d$, and $G=(\zz_p)^a$.  Every continuous $G$-equivariant map $f: G^{*(n+1)} \to W_G^{\oplus d} $ has a zero.
\end{theorem}

Note that the case $r=2$ above is exactly the Borsuk--Ulam theorem.

\begin{theorem}[Selection-structure Volovikov theorem]\label{thm:selection-volovikov}
    Let $d,m, a,q$ be positive integers, $p$ be a prime number, and $r=p^a$.  Assume $2 \le q \le r$.  Let $G= (\zz_p)^a$; note that $|G|=r$.  Let $n=(r-1)d$.  Let $W$ be a finite set and $\Sigma = (\kk, \LL)$ be an $n$-admissible selection structure on $W$.  For each $w \in W$, let $X^w = \{x^w_1,\dots, x^w_m\}$ be a set of $m$ points in $\rr^{d-1}$.  Finally, for each $w \in W$ and $i \in [m]$, let $A^w_i$ be a closed subset of $G^{*(n+1)}$.  Suppose that the following conditions hold.
    \begin{itemize}
        \item For all $W' \subseteq W$ such that $W \setminus W' \not\in \LL$, we have
        \[
        G^{*(n+1)} = \bigcup_{w \in W'} \bigcup_{i =1}^m \bigcup_{g \in G} gA_i^w.
        \]
        \item For every $q$-tuple of different elements $w_1, \dots, w_q$ of $W$ such that $\{w_1,\dots,w_q\} \in \kk$, every $q$-tuple $g_1,\dots,g_q$ of different elements of $G$, and every $i \in [m]$, we have $\bigcap_{j \in [q]} g_jA^{w_j}_i  = \emptyset$.
    \end{itemize}
    Then, there exists $x \in G^{*(n+1)}$ and subsets $P_g \subset W \times [m]$ for each $g \in G$ such that the following holds.  First, the projection onto the first coordinate $\pi_1: \bigcup P_g \to W$ is injective and $\pi_1 (\bigcup P_g)$ is a facet of $\kk$.  The projection $\pi_2: \bigcup P_g \to [m]$ satisfies that for any $q$ different sets in $\{P_g: g \in G\}$, their projections $\pi_2(P_g)$ have empty intersection.  Additionally, we have
\[
\bigcap_{g \in G}\conv\{x^w_i : (w,i) \in P_g\} \neq \emptyset \qquad \mbox{and}\qquad x \in \bigcap_{g \in G}\bigcap_{(w,i) \in P_g} g A^w_i.
\]
\end{theorem}

The condition $\bigcap_{g \in G}\conv\{x^w_i : (w,i) \in P_g\} \neq \emptyset$ implies that $P_g \neq \emptyset$ for all $g \in G$, which will be relevant in our applications.

    \begin{proof}
    Let $\delta>0$ be such that for all distinct $w_1,\dots, w_q \in W$ with $\{w_1,\dots,w_q\}\in \kk$, all distinct $q$-tuples $g_1,\dots,g_q$ of elements of $G$, and every $i \in [m]$, no set of $G^{*(n+1)}$ of diameter at most $\delta$ intersects all the sets $g_jA^{w_j}_i$ for $j \in [q]$.

    Let $T$ be a $G$-invariant triangulation of $G^{*(n+1)}$ with mesh strictly smaller than $\delta$, and let $E = G \times [m]$, with a natural action of $G$.

    For a vertex $v$ of $T$, let $D_w(v) =\{(g,i) : v \in gA^w_i \}$.  Let $U(v) = \{w \in W: D_w(v) \neq \emptyset\}$.  The set $U(v)$ is the set of elements $w \in W$ such that $v$ is contained in $\bigcup_{g\in G}g A^w_i$ for some $i \in [m]$.

    If we had $U(v) \not\in \LL$, notice that by taking $W' = W\setminus U(v)$, we would have
    \[
    v \not\in \bigcup_{w \in W'} \bigcup_{i=1}^{m} \bigcup_{g \in G} g A_i^w ,
    \]
    contradicting the assumptions.  Therefore, $U(v) \in \LL$ for all vertices $v$.

    For a simplex $\tau$ of $T$, let $V(\tau)$ be its set of vertices, and define
    \[
    D_w(\tau) = \bigcup_{v \in V(\tau)} D_w(v) \qquad U(\tau) = \bigcup_{v \in V(\tau)}U(v).
    \]
    Since $\LL$ is upward-closed and $U(v) \in \LL$ for all $v$, we have $U(\tau) \in \LL$ too.  Finally, let $D(\tau) = \{D_w(\tau) : w \in U(\tau)\}$.  Consider the set
    \[
    \Gamma (\tau) = \kk (U(\tau), D(\tau)).
    \]
    This is the parallel expansion of $\kk[U(\tau)]$ by the fibers $D(\tau)$.  Note that for $\sigma \subseteq \tau$, we have $\Gamma (\sigma) \subseteq \Gamma (\tau)$, and $\Gamma( h \tau) = h \Gamma (\tau)$ for all $h \in G$.

    The goal now is to construct a map
    \begin{align*}
        \psi : G^{*(n+1)}=|T| & \to |\kk(W; E)| \qquad \mbox{such that for all } \tau \in T \\
        \psi (\tau) & \subseteq |\Gamma(\tau)|
    \end{align*}

    This is done inductively on skeleta.  For each vertex $v$ of $T$, we define $\psi(v)$ as a vertex of $\Gamma(v)$ arbitrarily and equivariantly with respect to the $G$-action.  If the map $\psi$ has been defined on the $(k-1)$-skeleton and $\tau$ is a $k$-simplex, notice that $\psi (\partial \tau) \subset |\Gamma(\tau)|$.  Since $\Gamma(\tau)$ is $(n-1)$-connected, the map $\psi$ can be extended over $\tau$, as long as $k \le n$.  We extend the map on $g\tau$ for all $g \in G$ equivariantly with respect to the $G$-action.

    Now we define a second map
    \begin{align*}
        \Phi: |\kk(W; E)| & \to W_G \otimes \rr^{d} = W_G \otimes (\rr^{d-1}\times \rr )
    \end{align*}
    by defining it on the vertices.  To do this, let $\{\lambda_g: g\in G\}$ be the set of vertices of a regular simplex in $W_G$ centered at the origin.  For example, if $\{e_g \in \rr^G\}$ is the $\{0,1\}$-vector with coordinate $1$ only at entry $g$, we can take $\lambda_g = e_g - \frac{1}{r}\sum_{h \in G}e_h$.  Recall that the vertices of $\kk(W;E)$ are equal to the set $W \times G \times [m]$, due to the choice of fibers $E$.  Now we define
    \[
    \Phi (w,g,i) = \lambda_g \otimes (x^w_i,1)
    \]
    and extend linearly on the faces of $|\kk(W; E)|$.  The map \(F= \Phi \circ \psi : G^{*(n+1)}  \to (W_G)^{\oplus d}\) is continuous and $G$-equivariant.  By Volovikov's theorem, it must have a zero.

    Let $z$ be a zero of $F$.  Let $\tau$ be a face of $T$ that contains $z$, and $\eta$ be a containment-maximal face of $\Gamma(\tau)$ that contains $\psi (z)$.  Let $P_g = \{(w,i) \in W \times [m]: (w,g,i) \in \eta \}$.

    The core of the tensoring technique from B\'ar\'any and Onn \cite{Barany1997} (based on Sarkaria's proof of Tverberg's theorem \cite{Sarkaria1992}) is that
    $0 \in \conv \Phi(\eta)$ if and only if
    \[
    \bigcap_{g \in G}\conv \{x^w_i: (w,i) \in P_g\} \neq \emptyset.
    \]
    See \cite{Sarkar2022} for a standard deduction.  The key reason is that the only nontrivial linear dependence, up to scalar multiplication of the set $\{\lambda_g: g \in G\}$ is $\sum_{g \in G}\lambda_g = 0$.  In other words, if a linear combination of the vectors $\lambda_g$ is zero, their coefficients must be equal.  This carries through the tensor product.  If $0 \in \conv \Phi(\eta)$, we can factor the elements tensored with $\lambda_g$ for each $g$, which, after normalizing the coefficients, leads us to a point in $\conv \{x^w_i: (w,i) \in P_g\}$.  Those points must be equal due to the properties of $\lambda_g$ described above.

    Now let us prove the properties of $\pi_1: (\bigcup_{g\in G} P_g)\to W$.  Since the faces of $|\kk(W; E)|$ have at most one vertex from each fiber, $\pi_1$ is injective.  We also have $\pi_1(\bigcup_{g\in G}P_g)= \pi_1(\eta)$, which by definition is a facet of $\kk[U(\tau)]$.  By the second condition for selection structures, this is a facet of $\kk$.

    Now let us show that for any $q$ pairwise distinct elements $g_1,\dots,g_q$ in $G$, we have $\bigcap_{j\in [q]}\pi_2(P_{g_j}) = \emptyset$.  If we had $(u,g,i) \in \eta$ and $(w, h,i) \in \eta$, we would first have $u \neq w$.  Therefore, if we have $i \in \bigcap_{j \in [q]} \pi_2(P_{g_j})$, we would need $q$ different elements $w_1,\dots, w_q$ of $W$ such that $(w_j,g_j,i) \in \eta$ for all $j \in [q]$.  This implies $\{w_1,\dots,w_q\} \in \kk$.  Since $\eta \subseteq \Gamma(\tau)$, that means there are vertices $v_1,\dots,v_q \in V(\tau)$ with $v_j \in g_j A^{w_j}_i$ for all $j \in [q]$.  However, the diameter of $\tau$ is less than $\delta$, contradicting the choice of $\delta$.  Therefore no index $i$ appears in all sets $\pi_2(P_{g_j})$ for all $j \in [q]$.


    
    What we have shown is that, for every $(w,i) \in \bigcup_{g\in G} P_g $, some vertex of $\tau$ is in the corresponding set $gA^{w}_i$ (depending on which $P_g$ contains $(w,i)$).  If we take a sequence of triangulations $T_{\alpha}$ with mesh tending to $0$, the finiteness of $W$ and the compactness of $G^{*(n+1)}$ imply by a standard argument that the same holds for some point $x$ in the limit, which gives us the desired conclusion
    \[
    x \in \bigcap_{g \in G}\bigcap_{(w,i)\in P_g}gA^w_i.
    \]

\end{proof}

Just as special configurations of points can have very restricted Radon partitions, the same can be said for Tverberg partitions.  Given a set $G$ with $r$ elements and $n=d(r-1)$, we say that a sequence $g_1,\dots, g_{n+1}$ of elements of $G$ is rainbow if the set $\{g_j : (i-1)(r-1)+1 \le j \le i(r-1)+1\}$ is equal to $G$ for all $i=1,\dots, d$.  When $|G|=2$, rainbow sequences are simply alternating sequences.  Bukh, Loh, and Nivasch showed that for every $m$, we can find ordered $m$-point sequences $x_1,\dots,x_m$ in $\rr^{d-1}$ such that the only minimal Tverberg partitions they form correspond to rainbow subsequences \cite{Bukh2017}.  This allows us to prove the following version of Fan's theorem.  See also the work of P\'or on unavoidable Tverberg sequences \cite{por2018arxiv} and the work of Frick and Jeffs on $d$-Tverberg complexes \cite{Frick2024}.  Notice that the case $r=2$ below is exactly \cref{thm:fan-1952}.

\begin{corollary}[Fan-type covering version of Volovikov's theorem]
    Let $d,m,a$ be positive integers, $p$ be a prime, let $r=p^a$, and let $n=(r-1)d$.  Let $G = (\zz_p)^a$.  Let $A_1,\dots, A_m$ be closed subsets of $G^{*(n+1)}$.  Suppose that $G^{*(n+1)} = \bigcup_{i=1}^m \bigcup_{g \in G} gA_i$, and that for any distinct $g,h \in G$ and every $i \in [m]$ we have $gA_i \cap hA_i = \emptyset$.  Then, we can find indices $i_1<\dots< i_{n+1}$ of $[m]$ and a rainbow sequence $g_1,\dots,g_{n+1}$ of elements of $G$ such that
    \[
    \bigcap_{j=1}^{n+1} g_{j}A_{i_j} \neq \emptyset.
    \]
\end{corollary}

\begin{proof}
    Consider the selection structure on $W=[n+1]$, $\kk = 2^W$, and $\LL = \{W\}$.  For each $w \in W$, let $A^w_i = A_i$.  Let $X=\{x_1,\dots,x_m\}$ be a sequence of points in $\rr^{d-1}$ from the constructions by Bukh, Loh, and Nivasch (i.e., a monotone sequence in the stretched grid of $\rr^{d-1}$).  The minimal Tverberg partitions of $X$ into $|G|=r$ sets correspond exactly to the rainbow sequences for the subsets of size $n+1$.  For each $w \in W$, let $X^w = X$.  With this setup, apply \cref{thm:selection-volovikov} with $q=2$.  If any index $i$ appears more than once in a set $P_g$, delete all but one occurrence.  This does not break the convex hull intersection property in the conclusion.  The choice of points $X$ guarantees the desired rainbow sequence in the conclusion.
\end{proof}

\section{Applications to fair partitions}\label{sec:applications}

The Borsuk--Ulam theorem and Volovikov's theorem are widely applied in the context of fair partitions and mass partitions \cite{RoldanPensado2022}.  Consider, for instance, the ham sandwich theorem.  We say that a measure on $\rr^d$ is absolutely continuous if it is absolutely continuous with respect to the Lebesgue measure on $\rr^d$.

\begin{theorem}[Banach 1938 \cite{Steinhaus1938}]
    Let $d$ be a positive integer and let $\mu_1,\dots,\mu_d$ be absolutely continuous probability measures on $\rr^d$.  Then, there exists a hyperplane $H$ such that its two closed half-spaces $H^+$ and $H^-$ satisfy $\mu_i(H^+) = \mu_i(H^-)$ for all $i$.
\end{theorem}

The general version is due to Stone and Tukey \cite{Stone1942}.  The number of measures is optimal.  If we are given $m>d$ measures on $\rr^d$, we can no longer guarantee that there exists a hyperplane simultaneously halving all of them.  However, as Frick and Wellner showed, the extensions of Fan's theorem give us interesting consequences even in this case.

We describe below the selection-structure consequences for the ham sandwich theorem.  However, the methods below have much more reach than just this application.  Almost any mass partition problem in which the space of partitions can be parametrized by $S^d$ has a selection-structure version of this form.  This includes, for example, the necklace splitting problem for two thieves (see \cite{Alon1986} for the proof using the reduction to Borsuk--Ulam).  

More generally, mass partition problems into $r\ge 2$ parts, which can be  parametrized by $G^{*n}$ where $G$ is a discrete set with $r$ elements, can be extended using \cref{thm:selection-volovikov}.  This includes, for example, the general necklace splitting problem \cite{Alon1987} and results using hyperplane arrangements with fixed directions \cite{Karasev2016}.  One important distinction is that the usual reduction from arbitrary $r$ to prime-powers does not obviously preserve the additional selection-structure constraints.  Due to this, we restrict our results to the prime-power case.

Let us first state and prove the selection-structure ham sandwich theorem.

\begin{theorem}[Selection-structure ham sandwich theorem]
    Let $m,d$ be positive integers, let $W$ be a finite set, and $\Sigma = (\kk,\LL)$ be a $d$-admissible selection structure on $W$.  For each $w \in W$, let $\{x^w_1,\dots,x^w_m\} \subset \rr^{d-1}$ be a set of $m$ points.  Let $\mu_1,\dots,\mu_m$ be absolutely continuous probability measures on $\rr^d$.  Assume there does not exist a hyperplane simultaneously bisecting all measures $\mu_i$ with $i \in [m]$.

    For each $w \in W$, let $G_w$ be a family of hyperplanes, closed under the standard topology of the affine Grassmannian of hyperplanes in $\rr^d$.  Suppose that for every $W' \subseteq W$ such that $W \setminus W' \not\in \LL$ and every hyperplane $H$ in $\rr^d$, we have $H \in \bigcup_{w \in W'}G_w$.

   Then, there exist a closed half-space $H^+$ and a facet $B \in \kk$ such that $H = \partial H^+ \in \bigcap_{w \in B}G_w$ and a function $f: B \to [m]$ with the following properties:

   \[
   |\mu_{f(w)}(H^+) - 1/2| = \max_{j \in [m]} |\mu_j(H^+) - 1/2| \qquad \mbox{ for all }w \in B. 
   \]
   Moreover, if we let $P_+ = \{w \in B: \mu_{f(w)}(H^+) - 1/2 >0\}$ and $P_- = \{w \in B: \mu_{f(w)}(H^+) - 1/2 <0\}$, then $f(P_+) \cap f(P_-) = \emptyset$ and
   \[
   \conv\{x^w_{f(w)} : w \in P_+\} \cap \conv\{x^w_{f(w)} : w \in P_-\} \neq \emptyset.
   \]
\end{theorem}

\begin{proof}
    The space of closed half-spaces in $\rr^d$ together with $\rr^d$ and $\emptyset$ can be parametrized by $S^d$ so that complementary half-spaces correspond to antipodal points.  For $w \in W$ and $i \in [m]$, let $A^w_i$ be the union of $\{\rr^d\}$ with the set of half-spaces $H^+$ such that their bounding hyperplane is in $G_w$ and such that
    \[
    \mu_i (H^+) - 1/2 = \max_{j \in [m]} |\mu_j(H^+) -1/2|.
    \]
    Let us show that the sets $A^w_i$ satisfy the conditions of \cref{thm:main-selection-frickwellner}.  First, given $W' \subseteq W$ such that $W \setminus W' \not\in \LL$ and a half-space $H^+$, let $H$ be its bounding hyperplane.

    We know that $H \in \bigcup_{w \in W'}G_w$.  For this hyperplane, let $i$ be an index such that $|\mu_i (H^+) - 1/2| = \max_{j \in [m]}|\mu_j (H^+) - 1/2|$.  If $\mu_i(H^+) > 1/2$, then $H^+ \in A^w_i$.  Otherwise, $H^+ \in - A^w_i$.  This implies the first property for these sets.

    To prove the second property, assume for the sake of a contradiction that we have $\{u,w\} \in \kk$ and $i \in [m]$ such that $A^u_i \cap -A^w_i \neq \emptyset$.  A half-space $H^+$ in this intersection would require that its bounding hyperplane is in $G_u \cap G_w$.  Additionally, to be in $A^u_i$ we need $\mu_{i}(H^+) -1/2 = \max_{j \in [m]}|\mu_j(H^+) - 1/2|$.  However, to be in $-A^w_i$ we would need that for $H^-$, the other closed half-space of $H$, we have
    \(\mu_i (H^-) - 1/2 = \max_{j \in [m]}|\mu_j (H^-) - 1/2|\).  This is impossible since $\mu_i (H^-) - 1/2 = -(\mu_i (H^+) - 1/2)$, so they cannot both be positive.

    We can therefore apply \cref{thm:main-selection-frickwellner}.  The conclusion of \cref{thm:main-selection-frickwellner} corresponds exactly to the desired conclusion.
\end{proof}


To illustrate the applications of \cref{thm:selection-volovikov}, we state and prove the selection-structure necklace splitting theorem.  The interpretation is that $r$ thieves want to split an open necklace with $m$ kinds of pearls.  They are willing to make $(r-1)d$ cuts to the necklace, for some integer $d$.  If $d=m$, the necklace splitting theorem states that we can find a set of cuts and a distribution of the resulting segments so that each thief gets the same amount of each kind of pearl.  Our result says that for $d<m$, we can still find cuts and a distribution that carry interesting information about the amount of pearls that some thieves receive.

We work with the continuous version, in which instead we have $m$ absolutely continuous probability measures on the interval $[0,1]$.  Given a group $G$ with $r$ elements, we can use $G$ to index the thieves.  Consider $Y$ to be a partition of $[0,1]$ into $n+1$ intervals for some $n \ge 0$, together with a distribution of the intervals among the $r$ thieves.  For $g \in G$, we can define $gY$ to be the same partition, but the parts that were given to thief $h$ now go to thief $(gh)$.  With this action on the distributions, we are now ready to state our version of the necklace splitting theorem.

The space of partitions of $[0,1]$ into $n+1$ intervals, each carrying one of $r$ possible labels, is naturally parametrized by $G^{*(n+1)}$, where $G$ is a discrete set with $r$ elements.  The diagonal action of $G$ on $G^{*(n+1)}$ is equivalent to the action described above.

This is because the element $\sum_{i=1}^{n+1}\alpha_i g_i \in G^{*(n+1)}$ corresponds to the partition into $n+1$ intervals, where the $i$-th interval has length $\alpha_i$ and label $g_i \in G$.  The label of $0$-length intervals does not matter.  Under this parametrization, we can define a closed or open set of partitions with labels.

\begin{theorem}
   Let $m,d, a$ be positive integers, let $p$ be a prime, and $r=p^a$, $n=(r-1)d$, $G = (\zz_p)^a$.  Let $W$ be a finite set and $\Sigma = (\kk,\LL)$ be an $n$-admissible selection structure on $W$.  For each $w \in W$, let $\{x^w_1,\dots,x^w_m\} \subset \rr^{d-1}$ be a set of points.  Let $\mu_1,\dots,\mu_m$ be absolutely continuous probability measures on the interval $[0,1]$.  Suppose no such labeled partition gives every thief exactly $1/r$ of every measure.

  Suppose we are given, for each  $w \in W$, a closed set $X_w$ of partitions of $[0,1]$ into $n+1$ intervals, each with a label in $G$.  Additionally, for each $W' \subseteq W$ such that $W \setminus W' \not\in \LL$, we have that every labeled partition is contained in $\bigcup_{w \in W'} \bigcup_{g \in G} g X_w$.

  Then, there exists a facet $B$ of $\kk$ and a partition of $[0,1]$ with a distribution $Y$, and functions $f: B \to [m]$, $h: B \to G$ so that the following holds.  Denote by $Y_g$ the union of the set of intervals assigned to thief $g$ by $Y$. 
  
  First, $Y \in \bigcap_{w \in B}\left( \bigcup_{g\in G}gX_w\right)$.  Second, for every $w \in B$ we have
  \[
  \mu_{f(w)}(Y_{h(w)}) = \max_{\substack{i \in [m]\\g \in G}}\left(\mu_i(Y_g) \right)>\frac{1}{r}.
  \]
  Third, if we let $P_g = \{w \in B: h(w) = g\}$, then $P_g$ is nonempty for all $g$, $\bigcap_{g \in G}f(P_g) =\emptyset$ and 
  \[
  \bigcap_{g \in G}\conv \{x^w_{f(w)}: w \in P_g\}\neq \emptyset.
  \]
\end{theorem}

\begin{proof}
    Let $e$ denote the identity element of $G$.  For
    $Y\in G^{*(n+1)}$, $i\in[m]$, and $g\in G$, let $Y_g$ be the union of all intervals with label $g$.  Define
    \[
    M(Y)=
    \max_{\substack{i\in[m]\\g\in G}}\left(\mu_i(Y_g)-\frac{1}{r}\right).
    \]
    The absolute continuity of the measures implies that the
    functions $\varphi_i^g$ are continuous.  Therefore, $M$ is
    continuous and $G$-invariant.

    For every $i\in[m]$, we have
    \[
    \sum_{g\in G}\left( \mu_i(Y_g)-\frac{1}{r} \right)=0.
    \]
    In particular, $M(Y)\ge 0$.  If $M(Y)=0$, then $Y$ would be
    an exact fair splitting, contrary to the assumptions.  Therefore,
    \[
    M(Y)>0
    \qquad\mbox{for every }Y\in G^{*(n+1)}.
    \]

    For every $w\in W$ and $i\in[m]$, let
    \[
    A_i^w=
    \left\{Y\in\bigcup_{g\in G}gX_w: \mu_i(Y_e)-\frac{1}{r}=M(Y)\right\}.
    \]
    Since $\bigcup_{g\in G}gX_w$ is a finite union of closed sets, it is
    closed.  Therefore, every $A_i^w$ is a closed subset of
    $G^{*(n+1)}$.

    Let us show that these sets satisfy the conditions of
    \cref{thm:selection-volovikov} with $q=r$.  First, let
    $W'\subseteq W$ satisfy $W\setminus W'\notin\LL$, and let
    $Y\in G^{*(n+1)}$.  By the covering assumption, there is some
    $w\in W'$ such that $Y\in\bigcup_{g \in G}g{X}_w$.  Choose $i\in[m]$ and
    $h\in G$ such that \(\mu_i(Y_h)-\frac{1}{r}=M(Y)\).
    We have 
    $h^{-1}Y\in\bigcup_{g\in G}g{X}_w$.  Moreover,
    \[
    \mu_i((h^{-1}Y)_e-\frac{1}{r}
    =
    \mu_i(Y_h)-\frac{1}{r}
    =
    M(Y)
    =
    M(h^{-1}Y).
    \]
    Consequently, $h^{-1}Y\in A_i^w$, and therefore
    $Y\in hA_i^w$.  This proves that
    \(
    G^{*(n+1)}
    =
    \bigcup_{w\in W'}\bigcup_{i=1}^m\bigcup_{g\in G}gA_i^w
    \).

    Now let $w_g$, for $g\in G$, be distinct elements of $W$ such
    that
    \(
    \{w_g:g\in G\}\in\kk
    \),
    and let $i\in[m]$.  Suppose, for the sake of a contradiction,
    that there is a configuration
    \[
    Y\in\bigcap_{g\in G}gA_i^{w_g}.
    \]
    For every $g\in G$, we then have $g^{-1}Y\in A_i^{w_g}$, and
    therefore \(
    \mu_i(Y_g)-\frac{1}{r}=M(Y)
    \).
    It follows that
    \[
    0
    =
    \sum_{g\in G}
    \left(\mu_i(Y_g)-\frac{1}{r}\right)
    =
    rM(Y),
    \]
    contradicting the fact that $M(Y)>0$.  Therefore,
    \[
    \bigcap_{g\in G}gA_i^{w_g}=\emptyset.
    \]
    This proves the second condition of
    \cref{thm:selection-volovikov} for $q=r$.

    We can therefore apply \cref{thm:selection-volovikov}.  Let
    $Y\in G^{*(n+1)}$ and let $Q_g\subseteq W\times[m]$, for
    $g\in G$, be the sets obtained from that theorem.  Let
    \[
    B=\pi_1\left(\bigcup_{g\in G}Q_g\right).
    \]
    Then, $B$ is a facet of $\kk$, and the injectivity of $\pi_1$
    allows us to define functions $f:B\to[m]$ and $h:B\to G$ by
    declaring that
    \[
    f(w)=i
    \qquad\mbox{and}\qquad
    h(w)=g
    \]
    whenever $(w,i)\in Q_g$.

    Let $P_g=\{w \in B: h(w) = g\}$.  If $w\in P_g$, then
    $(w,f(w))\in Q_g$, and the conclusion of
    \cref{thm:selection-volovikov} gives \(Y\in gA_{f(w)}^w\).  This implies
    $Y\in\bigcup_{g\in G}gX_w$.  Therefore,
    \[
    Y\in\bigcap_{w\in B}\bigcup_{g\in G}gX_w.
    \]
    We also have \( \mu_{f(w)}(Y_g)-\frac{1}{r} = M(Y) \)
    for every $w\in P_g$, which proves the second conclusion.

    Finally, notice that
    \(
    \pi_2(Q_g)=f(P_g)\).  
    The conclusion of \cref{thm:selection-volovikov} for $q=r$ gives
    \(
    \bigcap_{g\in G}f(P_g)=\emptyset\).  Its convexity conclusion implies
    \[
    \bigcap_{g\in G}
    \conv\{x^w_{f(w)}:w\in P_g\}\neq\emptyset.
    \]
\end{proof}

\section*{Acknowledgments}

The author thanks Florian Frick and Zoe Wellner for their comments.


\end{document}